\documentclass[11pt,letterpaper]{amsart}
\usepackage{fancyhdr}
\usepackage{bm}
\usepackage{hyperref}
\usepackage{graphicx} % Required for inserting images
\usepackage{nicematrix}
\usepackage{amsthm,amssymb,amsmath}
\usepackage[T1]{fontenc}
\usepackage[utf8]{inputenc}
\usepackage{hyperref}
\usepackage{lipsum}
\usepackage{mathrsfs}
\usepackage{enumitem}
\usepackage{stmaryrd}
\usepackage{setspace}
\usepackage{yfonts}
\usepackage{float}
\usepackage{mathtools}
\usepackage{parskip}
\usepackage[all,cmtip]{xy}
\usepackage{tikz-cd}
\tikzcdset{row sep/normal=50pt, column sep/normal=50pt}

\newtheorem{lemma}{Lemma}[section]
\newtheorem{remark}[lemma]{Remark}

\newtheorem{theorem*}{Theorem}

\newtheorem{example*}[lemma]{Example}

\newtheorem{conjecture}[lemma]{Conjecture}
\newtheorem{manualtheoreminner}{Theorem}
\newenvironment{manualtheorem}[1]{%
  \renewcommand\themanualtheoreminner{#1}%
  \manualtheoreminner
}{\endmanualtheoreminner}
\usepackage{blindtext}

\usepackage[backend=biber, style=numeric, sorting=nyt]{biblatex}
\title{Proof of Miyanishi's conjecture on endomorphisms of varieties}
\author{\small{Supravat Sarkar}}
\date{}
\begin{document}

\begin{abstract}
    If $X$ is a quasi-projective variety over a field $k$ and $\phi$ a birational endomorphism of $X$ that is injective outside a closed subset of codimension $\geq 2$, we prove that $\phi$ is an automorphism. This generalizes an old theorem of Ax and proves a conjecture of Miyanishi from 2005. A key step in our proof is a finiteness result on class groups, which is of interest in its own right.
\end{abstract}
\maketitle
\begin{center}
\end{center}
\begin{center}
\end{center}

\section{Introduction}

The study of endomorphisms of algebraic varieties is a very interesting and fruitful area of research. For normal projective varieties, existence of several natural classes of endomorphisms impose strong restrictions on the singularity and the global geometry of the variety, see \cite{fakhruddin2002questions},\cite{meng2020building}, \cite{meng2022non}, \cite{paranjape1989self}, \cite{totaro2023endomorphisms} and \cite{sarkar2025images}.

On the other hand, for endomorphisms of non-complete varieties, very little is known. Let $X$ be a reduced separated scheme of finite type over a field $k$. Ax proved in \cite{ax1969injective} that any injective endomorphism of $X$ is bijective. Later Nowak proved in \cite{nowak1994injective} that if $k=\overline{k}$ and char $k=0$, any bijective endomorphism of $X$ is an automorphism. So, when $k=\overline{k}$ and char $k=0$, any injective endomorphism of $X$ is an automorphism.

A natural question is whether the global injectivity hypothesis in the above statement can be weakened. Even for smooth $\mathbb{C}$-varieties, the endomorphism $\phi$ may not be an automorphism if we just assume $\phi$ is injective on a dense open set, as the example: $\phi:\mathbb{A}^2\to \mathbb{A}^2$ given by $\phi(x,y)=(x,xy)$ shows. The next question to ask is whether $\phi$ is an automorphism if $\phi$ is injective on an open set which is large enough in the sense that its complement has codimension $\geq 2$. In \cite{miyanishi2005open}, Miyanishi conjectured the answer to be yes. 
\begin{conjecture}\label{Miyanishi}
    Let $X$ be a $\mathbb{C}$-variety and $\phi: X\to X$ be an endomorphism that is injective outside a closed subset of codimension $\geq 2$ in $X$. Then $\phi$ is an isomorphism.
\end{conjecture}

Strictly speaking, there was another part in the Conjecture of Miyanishi, which was proved in {\cite[Theorem 1.2]{das2022endomorphisms}} using ideas of \cite{kaliman2005theorem}.

In \cite{kaliman2005theorem}, it is proved that Conjecture \ref{Miyanishi} is true for $X$ whenever it is true for the normalization of $X$. So we can assume $X$ to be normal. Under the hypothesis of Conjecture \ref{Miyanishi}, it is clear that $\phi$ is birational with exceptional locus having codimension $\geq 2$. So
Conjecture \ref{Miyanishi} is clearly true when $\dim X\leq 2$. By {\cite[Section II.4.4]{shafarevich2016basic}}, Conjecture \ref{Miyanishi} holds when $X$ is $\mathbb{Q}$-factorial.  Conjecture \ref{Miyanishi} is also known to be true when  $X$ is affine or complete by \cite{kaliman2005theorem}. To the best knowledge of the author, the above is essentially everything that is known about Conjecture \ref{Miyanishi} till now. Some other cases of Conjecture \ref{Miyanishi} appeared also in \cite{biswas2025injective}. But we had difficulty following their argument, as they crucially use {\cite[Corollary 4.9(e)]{kallstrom2013purity}} and {\cite[Second sentence in Remark 4.7]{kallstrom2013purity}}, which need additional assumptions, as stated in the recent version \cite{Kallstrom2010purity}. Some of these results were also used in \cite{asano2025endomorphisms}.

The goal of this paper is to prove Conjecture \ref{Miyanishi} for all quasi-projective varieties. We in fact prove the following Theorem, which can be thought of as a more general version of Conjecture \ref{Miyanishi} which works over fields of arbitrary chracteristic.
\begin{manualtheorem}{A}\label{main}
    Let $X$ be a quasi-projective variety over any field $k$, and $\phi:X\to X$ a birational morphism that is injective outside a closed subset of codimension $\geq 2$ in $X$. Then $\phi$ is an isomorphism.
\end{manualtheorem}
Our proof in fact works for a much larger class of varieties than quasi-projective ones, see Remark \ref{non qp}.

In characteristic $0$, the birationality of $\phi$ in Theorem \ref{main} is immediate from the injectivity hypothesis (for $k=\overline{k}$ this is standard, for $k\ne\overline{k}$ one can use \cite{poonen2001points}), but in positive characteristic this birationality assumption is needed for Theorem \ref{main} to be true, as the example of Frobenius map of $\mathbb{P}^n$ shows.

In the proof, we have to talk about varieties over non-algebraically closed fields. So let us fix our convention: for any field $k$, a \textit{variety} over $k$ (or a $k$-variety) is an integral separated scheme of finite type over $k$. Unless otherwise stated, a point of a variety will always mean a closed point. For a pure dimensional scheme $X$ of finite type over a field, we will denote the group of Weil divisors, that is, codimension $1$ cycles on $X$ by $\text{Div }X$, and the class group, that is, Chow group of codimension $1$ cycles by $\textrm{Cl } X$. If $X$ is a normal variety, we identify the Picard group $\textrm{Pic } X$ as a subgroup of $\textrm{Cl } X$.  For a birational map $f:Y\dashrightarrow X$ of normal varieties, $\text{Ex}(\phi)$ will denote the exceptional locus of $\phi$: the complement of the largest open subset on which $\phi$ is defined and gives an isomorphism onto its image. Also, for a scheme $Y$, we will denote the underlying topological space by $|Y|$. We call a field finitely generated if it is finitely generated over its prime field.

\section{Proof of Theorem \ref{main}}

The key steps of the proof are the following two Lemmas, which are also interesting on their own.
\begin{lemma}\label{fg}
    Let $X$ be a variety over a finitely generated field $k$. Then $\textrm{Cl }X$ is a finitely generated abelian group.
\end{lemma}
\begin{proof}
Let $F$ be the prime field of $k$. Since $k$ is a finitely generated extension of $F$, taking a model of $X$ over a finitely generated $F$-subalgebra of $k$ with fraction field $k$, we see that there is a dominant morphism of $F$-varieties $f:Y\to Z$ such that the function field of $Z$ is $k$ and generic fibre of $f$ is $X$. Restriction to generic fibre gives a surjection $\textrm{Cl }(Y)\twoheadrightarrow \textrm{Cl }(X)$, so it suffices to show $\textrm{Cl }(Y)$ is finitely generated. Hence we can assume without loss of generality that $k$ is a prime field.

   There is a normal projective variety $Y$ birational to $X$. If $\textrm{Cl }(Y)$ is finitely generated, repeated application of excision exact sequence {\cite[Proposition 1.8]{fulton2013intersection}} shows that $\textrm{Cl }(X)$ is finitely generated too. Thus we can assume without loss of generality that $X$ is normal projective. As $k$ is perfect, $X$ is geometrically normal.

   For an abelian variety $A$ over $k$, the group of $k$-rational points $A(k)$ is finitely generated: for $k$ a finite field this is finite and for $k=\mathbb{Q}$ this follows from Mordell-Weil Theorem. Now the Lemma follows from {\cite[Theorem 17]{kollar2018mumford}}.
\end{proof}
\begin{lemma}\label{Cl by Pic}
    Let $\phi:Y\to X$ be a birational morphism of normal quasi-projective varieties over a field $k$ with $\text{Ex}(\phi)$ having codimension $\geq 2$, and $\dim \frac{\textrm{Cl }X}{\textrm{Pic }X}\otimes \mathbb{Q}<\infty$. If $\phi$ is not an open immersion, then we have $$\dim \frac{\textrm{Cl }Y}{\textrm{Pic }Y}\otimes \mathbb{Q}<\dim \frac{\textrm{Cl }X}{\textrm{Pic }X}\otimes \mathbb{Q}.$$
\end{lemma}
\begin{proof}
  let $E=\text{Ex}(\phi)$ and $F=\text{Ex}(\phi^{-1})=X\setminus\phi(Y\setminus E)$. Let $\phi^*$ denote the composition: $$\text{Div}(X)\twoheadrightarrow \text{Div}(X\setminus F)\cong \text{Div}(Y\setminus E) \cong \text{Div}(Y).$$ Here the first surjection in induced by restriction, second isomorphism is induced by the isomorphism $Y\setminus E\xrightarrow{\phi} X\setminus F$, and the third isomorphism is induced by restriction, noting that codim $E\geq 2$. Note that $\phi^*$ induces a surjection $$ \overline{\phi^*}:\frac{\textrm{Cl }X}{\textrm{Pic }X}\otimes \mathbb{Q}\to \frac{\textrm{Cl }Y}{\textrm{Pic }Y}\otimes \mathbb{Q}.$$ It suffices to show that $\overline{\phi^*}$ is not injective, if $\phi$ is not an open immersion.
  
 As $\phi$ is not an open immersion, we have $E\neq \varnothing.$ As $X$ is normal, by Zariski's main theorem $\phi|_E$ is not quasi finite. So, there are two distinct points $x$ and $y$ in $E$ with $\phi(x)=\phi(y)$. As $X$ is quasi-projective, there is an effective Cartier divisor $H$ in $X$ with $x\in |H|$ but $y\not\in |H|$:  if $L$ is an ample line bundle on $X$, then we can take $H$ to be the zero scheme of a section of $L^N$ vanishing at $x$ but not vanishing at $y$, which exists for all sufficiently large integers $N$. Taking closures of the irreducible components of $\phi(H\setminus E)$, we get an effective Weil divisor $G$ in $\text{Div}(X)$ with $\phi^*(G)$ the Weil divisor associated to $H$. As $H$ is effective Cartier we see that $\overline{\phi^*}(\overline{G})=0$, where $\overline{G}$ denotes the class of $G$ in $\frac{\textrm{Cl }X}{\textrm{Pic }X}\otimes \mathbb{Q}$.  We proceed to show that $\overline{G}\neq 0$.
 
 Suppose $\overline{G}= 0$. So, $mG$ is Cartier for some integer $m>0$. In other words, there is an effective Cartier divisor $D$ in $X$ whose associated Weil divisor is $mG.$

As $\phi^{-1}(D)$ is effective Cartier and $\phi^{-1}(|D|)=|\phi^{-1}(D)|$, we see that $\phi^{-1}(|D|)$ has pure codimension $1$ in $Y$. So, $\phi^{-1}(|D|)=\overline{\phi^{-1}(|D|)\setminus E}=|H|$. As $y\not\in |H|$, we get $y\not\in \phi^{-1}(|D|)$, that is, $\phi(y)\not\in |D|$. As $\phi(x)=\phi(y)$, we have $\phi(x)\not\in |D|$. But as $x\in H$, so $\phi(x)$ is in the support of $D$, hence $\phi(x)\in |D|$, a contradiction.

\end{proof}
Now we are ready to prove Theorem $\ref{main}.$

\textit{Proof of Theorem \ref{main}:}
Let $Z\subset X$ be a closed subset of codimension $\geq 2$ such that $\phi_{X\setminus Z}$ is injective.
There is a finitely generated subfield $K$ of $k$ such that $X, Z, \phi$, and the data of a locally closed embedding of $X$ in a projective space are all defined over $K$. So replacing $k$ by $K$, we can assume $k$ is a finitely generated field. 

We can assume without loss of generality that $X$ is normal, by \cite{kaliman2005theorem}. So by Lemma \ref{fg}, we have $\dim \frac{\textrm{Cl }X}{\textrm{Pic }X}\otimes \mathbb{Q}<\infty$. Also, the injectivity hypothesis, together with Zariski's main theorem implies codim $\text{Ex}(\phi)\geq 2$. Hence by Lemma \ref{Cl by Pic}, $\phi$ must be an open immersion. Finally by Ax's theorem as mentioned in the introduction (see \cite{ax1969injective}), $\phi$ is surjective, therefore an isomorphism.
\begin{remark}\label{non qp}
   A closer inspection of the proof shows that the only thing we used about the quasi-projectivity of the variety $X$ is the following: given an irreducible curve $C$ in $X$ and a point $x\in C$, there is an effective Cartier divisor $H$ in $X$ with $x\in H$ but $C\not\subset H$. So, Theorem \ref{main} is true whenever $X$ is a variety, not necessarily quasi-projective, satisfying the above hypothesis. This hypothesis is satisfied for a much larger class of varieties than quasi-projective ones. For example, this is always satisfied when $X$ has a quasi-finite (not necessarily dominant) morphism to some $\mathbb{Q}$-factorial variety, this holds in particular when $X$ is a subvariety of some regular variety.
\end{remark}
 \section{Acknowledgement}
 I thank Prof. János Kollár for giving valuable ideas and insightful discussions. I also thank Prof. Jakub Witaszek for giving helpful comments and feedbacks about the paper.
\printbibliography
\vspace{40pt}
\begin{flushleft}
{\scshape Department of Mathematics, Fine Hall, Princeton University, Princeton, NJ 700108, USA}.

{\fontfamily{cmtt}\selectfont
\textit{Email address: ss6663@princeton.edu} }
\end{flushleft}
\end{document}